\documentclass{amsart}
\usepackage{amscd,amsmath,amssymb,amsthm,bm,cases,color,comment,mathrsfs}
\usepackage{graphicx}
\usepackage[all]{xy}

\theoremstyle{definition}
\newtheorem{definition}{Definition}[section]

\theoremstyle{plain}
\newtheorem{theorem}[definition]{Theorem}
\newtheorem{proposition}[definition]{Proposition}
\newtheorem{lemma}[definition]{Lemma}

\newtheorem{claim}[definition]{Claim}
\newtheorem{question}[definition]{Question}

\numberwithin{equation}{section}

\title[Embeddability of RAAGs on complements of linear forests]{Embeddability of right-angled Artin groups on the complements of linear forests}

\author[T.~Katayama]{Takuya Katayama}
\address{
(Takuya Katayama)
Department of Mathematics, 
Hiroshima University, 
1-3-1 Kagamiyama, Higashi-Hiroshima, Hiroshima 739-8526, Japan 
}
\email{tkatayama@hiroshima-u.ac.jp}

\date{\today}

\begin{document}

\begin{abstract}
In this article, we prove that embeddings of right-angled Artin group $A_1$ on the complement of a linear forest into another right-angled Artin group $A_2$ can be reduced to full embeddings of the defining graph of $A_1$ into the extension graph of the defining graph of $A_2$. 
\end{abstract}

\maketitle

%%%%%%%%%%%%%%%%%
%%%%%Introduction%%%%%%
%%%%%%%%%%%%%%%%%
\section{Introduction}\label{Introduction}
Let $\Gamma$ be a simple graph (abbreviated a graph). 
We denote the vertex set and the edge set of $\Gamma$ by $V(\Gamma)$ and $E(\Gamma)$, respectively. 
The {\it right-angled Artin group} on $\Gamma$ is the group given by the following presentation:
$$
A(\Gamma) = \langle V(\Gamma) \mid \  v_i v_j v_i^{-1} v_j^{-1} = 1  \ \mbox{if} \ \{ v_i,v_j \} \in E(\Gamma) \ \rangle
.$$ 
A {\it graph homomorphism} is a map between the vertex sets of two graphs, which maps adjacent vertices to adjacent vertices. 
An injective graph homomorphism (abbreviated an embedding) $\iota \colon \Lambda \rightarrow \Gamma$ is called {\it full} if $\iota$ maps non-adjacent vertices to non-adjacent vertices. 
If there is a full embedding $\iota \colon \Lambda \rightarrow \Gamma$, then we denote by $\Lambda \leq \Gamma$ and $\Lambda$ is called a {\it full subgraph} of $\Gamma$. 
For finite graphs $\Lambda$ and $\Gamma$, it is well-known that $A(\Lambda)$ is isomorphic to $A(\Gamma)$ as a group if and only if $\Lambda$ is isomorphic to $\Gamma$ as a graph. 
Following S.~Kim and T.~Koberda \cite{Kim--Koberda13}, the {\it extension graph} $\Gamma^e$ of a finite graph $\Gamma$ is the graph such that the vertex set of $\Gamma^e$ consists of the words in $A(\Gamma)$ that are conjugate to the vertices of $\Gamma$, and two vertices of $\Gamma^e$ are joined by an edge if and only if they are commutative as words in $A(\Gamma)$. 
A celebrated theorem due to Kim--Koberda states that, if a finite graph $\Lambda$ is a full subgraph of the extension graph $\Gamma^e$ of a finite graph $\Gamma$, then we have an injective homomorphism (abbreviated an embedding) $A(\Lambda) \hookrightarrow A(\Gamma)$. 
In this article, a {\it path graph} $P_n$ on $n \ (\geq 1)$ vertices is the graph whose underlying space is homeomorphic to the origin $\{ 0 \}$ or unit interval $[0, 1]$ in the $1$-dimensional Euclidean space. 
A {\it linear forest} is the disjoint union of finitely many path graphs. 
The {\it complement} $\Lambda^c$ of a graph $\Lambda$ is the graph consisting of the vertex set $V(\Lambda^c) = V(\Lambda)$ and the edge set $E(\Lambda^c)= \{ \{ u, v \} \mid u,v \in V(\Lambda), \ \{ u , v \} \notin E(\Lambda) \}$. 

In \cite{Katayama17} the author ``proved" the following theorem. 

\begin{theorem}
Let $\Lambda$ be the complement of a linear forest and $\Gamma$ a finite graph. 
If $A(\Lambda) \hookrightarrow A(\Gamma)$, then $\Lambda \leq \Gamma$. 
\label{Main-theorem} 
\end{theorem}

We remark that Theorem \ref{Main-theorem} is equivalent to \cite[Theorem 1.3(1)]{Katayama17}. 
In the proof of \cite[Theorem 1.3(1)]{Katayama17}, the author used the following ``Theorem" (see the second line of the proof of Theorem 3.6 in \cite{Katayama17}). 

\begin{theorem}[{\cite[Theorem 3.14]{Casals-Ruiz15}}]
Let $\Lambda$ be the complement of a forest and $\Gamma$ a finite graph. 
If $A(\Lambda) \hookrightarrow A(\Gamma)$, then $\Lambda \leq \Gamma^e$. 
\label{Casals-Ruiz_lemma}
\end{theorem}

However, E.~Lee and S.~Lee \cite{Lee--Lee17} pointed out that the above ``Theorem" is incorrect by giving a counter-example. 
Thus the author's proof of Theorem \ref{Main-theorem} in \cite{Katayama17} is not valid. 

The purpose of this article is to give a complete proof of Theorem \ref{Main-theorem} by establishing the following theorem which shows that ``Theorem \ref{Casals-Ruiz_lemma}" holds when $\Lambda$ is the complement of a linear forest. 

\begin{theorem}
Let $\Lambda$ be the complement of a linear forest and $\Gamma$ a finite graph. 
If $A(\Lambda) \hookrightarrow A(\Gamma)$, then $\Lambda \leq \Gamma^e$. 
\label{linear-forest_lemma}
\end{theorem}

In fact, the author applied ``Theorem \ref{Casals-Ruiz_lemma}" only for the complement of linear forests in the proof of 
Theorem 1.3(1) in \cite{Katayama17}. 

We note that this theorem gives a partial positive answer to the following question. 

\begin{question}[{\cite[Question 1.5]{Kim--Koberda13}}]
For which graphs $\Lambda$ and $\Gamma$ do we have $A(\Lambda) \hookrightarrow A(\Gamma)$ only if $\Lambda \leq \Gamma^e$? 
\end{question}

With regard to this question, the reader is referred to the introduction of the paper \cite[Question 1.5]{Lee--Lee17} due to Lee--Lee. 

This article is organized as follows. 
In Section \ref{prelim}, we introduce terminology and known results. 
For the sake of convenience, we discuss relation between graph-join (a certain graph operation) and embedding problems in Section \ref{graph-join-sec}. 
Section \ref{the-linear-forest-lemma-sec} is devoted to the proof of Theorem \ref{linear-forest_lemma}.

\subsection*{Acknowledgements}
The author would like to thank his supervisor, Makoto Sakuma, for helpful discussions and a number of improvements regarding this paper. 
The author thanks to Eon-kyung Lee and Sang-jin Lee for their encouragements. 

%%%%%%%%%%%%%%%%%
%%%%%Preliminaries%%%%%%
%%%%%%%%%%%%%%%%%
\section{Preliminaries \label{prelim}}

Suppose that $\Gamma$ is a graph. 
An element of $V(\Gamma) \cup V(\Gamma)^{-1}$ is called a {\it letter}. 
Any element in $A(\Gamma)$ can be expressed as a {\it word}, which is a finite multiplication of letters. 
Let $w= a_1 \cdots a_l$ be a word in $A(\Gamma)$ where $a_1, \ldots, a_l$ are letters. 
We say $w$ is {\it reduced} if any other word representing the same element as $w$ in $A(\Gamma)$ has at least $l$ letters. 
The following lemma is useful for checking whether a given word is reduced or not (cf. \cite[Section 5]{Crisp--Wiest04}). 

\begin{lemma}
Let $w$ be a word in $A(\Gamma)$. 
Then $w$ is reduced if and only if $w$ does not contain a word of the form $v^{\epsilon}xv^{- \epsilon}$, where $v$ is a vertex of $\Gamma$, $\epsilon = \pm 1$ and $x$ is a word such that $v$ is commutative with all of the letters in $x$. 
\label{reduced}
\end{lemma}

The {\it support} of a reduced word $w$ is the smallest subset $S$ of $V(\Gamma)$ such that each letter of $w$ is in $S$ or $S^{-1}$; we write $S= \mathrm{supp}(w)$. 
It is well-known that the support does not depend on the choice of a reduced word, and so we can define the support of an element of $A(\Gamma)$. 
By a {\it clique}, we mean a complete subgraph of a graph. 
We rephrase a special case of Servatius' Centralizer Theorem {\cite[The Centralizer Theorem in Section III]{Servatius89}} as follows. 

\begin{lemma}
Let $w_1$ and $w_2$ be reduced words in $A(\Gamma)$ whose supports span cliques in $\Gamma$. 
Then the words $w_1$ and $w_2$ are commutative if and only if $\mathrm{supp}(w_1)$ and $\mathrm{supp}(w_2)$ are contained in a single clique in $\Gamma$. 
\label{commutative}
\end{lemma}

In this article, we say that a homomorphism $\psi \colon A(\Lambda) \rightarrow A(\Gamma)$ between two right-angled Artin groups satisfies {\it (KK) condition}  if $\mathrm{supp}(\psi(v))$ consists of mutually adjacent vertices in $\Gamma$ for all $v \in V(\Lambda)$ (i.e., $\mathrm{supp}(\psi(v))$ spans a clique in $\Gamma$). 

\begin{theorem}[{\cite[Theorem 4.3]{Kim--Koberda13}}]
Suppose that $\Lambda$ and $\Gamma$ are finite graphs and $A(\Lambda) \hookrightarrow A(\Gamma)$. 
Then there is an embedding $\psi \colon A(\Lambda) \hookrightarrow A(\Gamma^e)$ such that $\psi$ satisfies (KK) condition. 
Namely, for all $v \in V(\Lambda)$, $\mathrm{supp}(\psi(v))$ consists of mutually adjacent vertices in $\Gamma^e$. 
\label{Kim--Koberda}
\end{theorem}

\section{Graph-join \label{graph-join-sec}}
In this section, we prove Proposition \ref{joining-EGP}, which says that, for two finite graphs $\Lambda$ and $\Gamma$ such that there is an embedding $A(\Lambda) \hookrightarrow A(\Gamma)$ satisfying condition (KK), the problem of finding a full embedding $\Lambda \rightarrow \Gamma$, with a certain restriction, is reduced to the corresponding problems for the ``join-components" of $\Lambda$. 
The (graph-){\it join} $\Lambda_1 * \cdots *\Lambda_m$ of graphs $\Lambda_1, \ldots, \Lambda_m$ is the graph obtained from the disjoint union $\Lambda_1 \sqcup \cdots \sqcup \Lambda_m$ by joining the vertices $u$ and $v$ for all $u \in V(\Lambda_i), v \in V(\Lambda_j)$ ($i\neq j$).   
In this article, we say that a graph $\Lambda$ is {\it irreducible} (with respect to join) if $\Lambda$ cannot be the join of two non-empty graphs. 
Any finite graph $\Lambda$ is the join of finitely many irreducible graphs. 
Indeed, this follows from the fact that $\Lambda = \Lambda_1 * \cdots * \Lambda_m$ if and only if $\Lambda^c = \Lambda_1^c \sqcup \cdots \sqcup \Lambda_m^c$. 
This fact also implies the following lemma. 

\begin{lemma}
A finite graph $\Lambda$ is irreducible if and only if $\Lambda^c$ is connected. 
In particular, if $\Lambda$ is an irreducible graph containing at least two vertices, then for any vertex $u \in V(\Lambda)$, there is a vertex $u' \in V(\Lambda)$ such that $u$ and $u'$ are non-adjacent. 
\label{irr-nonad}
\end{lemma}
\begin{proof}
Suppose that $\Lambda$ is an irreducible graph containing at least two vertices. 
Pick a vertex $u \in V(\Lambda)$. 
If $u$ does not have a non-adjacent vertex, then we have a decomposition $\Lambda= \{ u \} * \check{\Lambda}$, where $\check{\Lambda}$ is a full subgraph spanned by $V(\Lambda) \setminus \{ u \}$, a contradiction. 
\end{proof}

Besides, the right-angled Artin group on the join,  $A(\Lambda_1 * \cdots *\Lambda_m)$, is isomorphic to the direct product $A(\Lambda_1) \times \dots \times A(\Lambda_m)$.  
For simplicity, if $\psi \colon A(\Lambda) \rightarrow A(\Gamma)$ is a homomorphism, then by $\mathrm{supp}(\psi)$ we denote $\cup_{v \in V(\Lambda)}\mathrm{supp}(\psi(v))$. 

\begin{proposition}
Let $\Lambda$ be the join $\Lambda_1 * \cdots * \Lambda_m$ of finite irreducible graphs $\Lambda_1, \ldots, \Lambda_m$, and let $\Gamma$ be a finite graph. 
Suppose that the following conditions hold: 
\begin{enumerate}
 \item[(1)] There is an embedding $\psi \colon A(\Lambda) \hookrightarrow A(\Gamma)$ satisfying condition (KK). 
 \item[(2)] For each $1 \leq i \leq m$, there is a full embedding $\iota_i \colon \Lambda_i \rightarrow \Gamma$ with $\iota_i(\Lambda_i) \subset \mathrm{supp}(\psi_i)$, where $\psi_i$ is the restriction of $\psi$ to $A(\Lambda_i)$. 
 \label{joining-EGP}
\end{enumerate} 
Then there is a full embedding $\iota \colon \Lambda \rightarrow \Gamma$ with $\iota(\Lambda) \subset  \mathrm{supp}(\psi)$. 
\end{proposition}

We first prove this proposition in a special case. 

\begin{lemma}
Let $\Lambda_1$ be a finite irreducible graph containing at least two vertices, and let $\Lambda_2$ and $\Gamma$ be finite graphs. 
Suppose that the following conditions hold: 
\begin{enumerate}
 \item[(1)] There is an embedding $\psi \colon A(\Lambda_1 * \Lambda_2) \hookrightarrow A(\Gamma)$ satisfying condition (KK). 
 \item[(2)] For $i=1,2$, there are full embeddings $\iota_i \colon \Lambda_i \rightarrow \Gamma$ with $\iota_i(\Lambda_i) \subset \mathrm{supp}(\psi_i)$, where $\psi_i$ is the restriction of $\psi$ to $A(\Lambda_i)$. 
\end{enumerate} 
Then the map $\iota \colon \Lambda_1 * \Lambda_2 \rightarrow \Gamma$, defined by $\iota(v) = \iota_1(v)$ or $\iota_2(v)$ according to whether $v \in V(\Lambda_1)$ or $v \in V(\Lambda_2)$, is a full embedding with $\iota(\Lambda_1 * \Lambda_2) \subset  \mathrm{supp}(\psi)$. 
 \label{joining-EGP-pre}
\end{lemma}
\begin{proof}
We have only to prove: (i) $\iota_1 (\Lambda_1) \cap \iota_2(\Lambda_2) = \emptyset$ and (ii) $\forall u \in V(\Lambda_1),  \forall v \in V(\Lambda_2)$, $\iota_1(u)$ and $\iota_2(v)$ are adjacent in $\Gamma$. 
In fact (i) and (ii) imply that the map $\iota \colon \Lambda_1 * \Lambda_2 \rightarrow \Gamma$ is a full embedding. 
Moreover, the assumptions $\iota_i(\Lambda_i) \subset \mathrm{supp}(\psi_i)$ imply that the full embedding $\iota$ satisfies the desired property that $\iota(\Lambda_1 * \Lambda_2) \subset \mathrm{supp}(\psi)$. 

(i) Pick $u_1 \in V(\Lambda_1)$ and $u_2 \in V(\Lambda_2)$. 
Since $\Lambda_1$ is irreducible and has at least two vertices, $\Lambda_1$ has a vertex $u_1'$ which is non-adjacent to $u_1$ in $\Lambda_1$ by Lemma \ref{irr-nonad}. 
Since $u_1$ is not adjacent to $u_1'$ in $\Lambda_1$, and since $\iota_1 \colon \Lambda_1 \rightarrow \Gamma$ is full, $\iota_1(u_1)$ is not adjacent to $\iota_1(u_1')$ in $\Gamma$. 
On the other hand, we can prove that $\iota_2(u_2)$ is either identical with $\iota_1(u_1')$ or adjacent to $\iota_1(u_1')$ in $\Gamma$ as follows (and so $\iota_1(u_1) \neq \iota_2(u_2)$ in any case).  
By the assumptions that $\iota_i(\Lambda_i) \subset \mathrm{supp}(\psi_i)$ ($i= 1,2$), there are vertices $\bar{u}_1' \in V(\Lambda_1)$ and $\bar{u}_2 \in V(\Lambda_2)$ such that $\iota_1(u_1') \in \mathrm{supp}(\psi(\bar{u}_1'))$ and $\iota_2(u_2) \in \mathrm{supp}(\psi(\bar{u}_2))$. 
Moreover, since $\Lambda_1$ and $\Lambda_2$ are joined in $\Lambda_1 * \Lambda_2$, the image $\psi(\bar{u}_2)$ is commutative with $\psi(\bar{u}_1')$, and therefore $\mathrm{supp}(\psi(\bar{u}_2))$ and $\mathrm{supp}(\psi(\bar{u}_1'))$ are contained in a single clique by (KK) condition and Lemma \ref{commutative}. 
Thus $\iota_2(u_2)$ is either adjacent to $\iota_1(u_1')$ or identical with $\iota_1(u_1')$ in $\Gamma$. 

(ii) Pick $u_1 \in V(\Lambda_1)$ and $u_2 \in V(\Lambda_2)$. 
There are vertices $\bar{u}_1 \in V(\Lambda_1)$ and $\bar{u}_2 \in V(\Lambda_2)$ such that $\iota_1(u_1) \in \mathrm{supp}(\psi(\bar{u}_1))$ and $\iota_2(u_2) \in \mathrm{supp}(\psi(\bar{u}_2))$. 
Since $\Lambda_1$ and $\Lambda_2$ are joined in $\Lambda_1 * \Lambda_2$, $\psi(\bar{u}_1)$ and $\psi(\bar{u}_2)$ are commutative. 
Hence, $\mathrm{supp}(\bar{u}_1)$ and $\mathrm{supp}(\bar{u}_2)$ are contained in a single clique in $\Gamma$ by Lemma \ref{commutative}. 
Thus, $\iota_1(u_1)$ and $\iota_2(u_2)$ are adjacent in $\Gamma$. 
\end{proof}

By $K_n$, we denote the complete graph on $n$ vertices. 
The right-angled Artin group on $K_n$, $A(K_n)$, is isomorphic to $\mathbb{Z}^n$. 

\begin{lemma}
Let $\Gamma$ be a finite graph. 
Suppose that $\psi \colon A(K_n) \rightarrow A(\Gamma)$ is an embedding satisfying condition (KK). 
Then there is a full embedding $\iota \colon K_n \rightarrow \Gamma$ with $\iota (K_n) \subset \mathrm{supp}(\psi)$. 
\label{Abel}
\end{lemma}
\begin{proof}
Since $\psi$ satisfies condition (KK) and since $K_n$ is complete, $\mathrm{supp}(\psi)$ spans a clique on $l$ vertices of $\Gamma$ by Lemma \ref{commutative}. 
Hence, we have an embedding $A(K_n) \cong \mathbb{Z}^n \hookrightarrow \mathbb{Z}^l \hookrightarrow A(\Gamma)$. 
This implies $n \leq l$, and so we obtain an injective map $V(K_n) \rightarrow \mathrm{supp}(\psi)$, which induces a full embedding $\iota \colon K_n \rightarrow \Gamma$ with $\iota (K_n) \subset \mathrm{supp}(\psi)$. 
\end{proof}

\begin{proof}[Proof of Proposition \ref{joining-EGP}]
We may assume that each of $\Lambda_1, \ldots \Lambda_n$ is a singleton graph and each of $\Lambda_{n+1}, \ldots, \Lambda_{m}$ has at least two vertices. 
Put $\Lambda_0:= \Lambda_1 * \cdots * \Lambda_n$. 
Then $\Lambda_0$ is isomorphic to the complete graph on $n$ vertices, $K_n$. 
In addition, we can decompose $\Lambda$ into $\Lambda_0 * (*_{i= n+1}^m \Lambda_i)$. 
Hence, we have $A(\Lambda) = A(\Lambda_0) \times A(\Lambda_{n+1}) \times \cdots \times A(\Lambda_m)$. 
By restricting $\psi$ to the abelian factor $A(\Lambda_0)$, we obtain an embedding $\psi_0 \colon A(\Lambda_0) \hookrightarrow A(\Gamma)$ satisfying condition (KK). 
Therefore, by Lemma \ref{Abel}, we obtain a full embedding $\iota_0: \Lambda_0 \rightarrow \Gamma$ with $\iota_0(\Lambda_0) \subset \mathrm{supp}(\psi_0)$. 
Consider the family of full embeddings $\iota_0, \iota_{n+1}, \ldots, \iota_{m}$. 
Since each of $\Lambda_{n+1}, \ldots, \Lambda_{m}$ is irreducible and has at least two vertices, by repeatedly applying Lemma \ref{joining-EGP-pre}, we obtain the desired full embedding $\iota \colon \Lambda_0 * \Lambda_{n+1} * \cdots * \Lambda_{m} \rightarrow \Gamma$. 
\end{proof}

\section{Proof of Theorem \ref{linear-forest_lemma} \label{the-linear-forest-lemma-sec}}
In this section we prove Theorem \ref{linear-forest_lemma}. 
We first rephrase Theorem \ref{linear-forest_lemma} in terms of join. 
Recall that $(\Lambda_1 \sqcup \cdots \sqcup \Lambda_m)^c = \Lambda_1^c * \cdots * \Lambda_m^c$. 
Hence, if $\Lambda$ is the complement of a linear forest, namely, if $\Lambda$ is the complement of the disjoin union of finitely many path graphs, then $\Lambda$ is the join of the complements of finitely many path graphs. 

\begin{theorem}[rephrased]
Let $\Lambda$ be the join of the complements of finitely many path graphs and $\Gamma$ a finite graph. 
If $A(\Lambda) \hookrightarrow A(\Gamma)$, then $\Lambda \leq \Gamma^e$. 
\label{linear-forest_lemma_reph}
\end{theorem}

To obtain a full embedding $\Lambda \rightarrow \Gamma^e$ in the assertion above, we consider the join-component, the complement $P_n^c$ of the path graph $P_n$ on $n$ vertices. 

\begin{lemma}
Let $n$ be a positive integer other than $3$ and $\Gamma$ a finite graph. 
Suppose that $\psi \colon A(P_n^c) \rightarrow A(\Gamma)$ is an embedding satisfying condition (KK). 
Then there is a full embedding $\iota \colon P_n^c \rightarrow \Gamma$ with $\iota(v) \in \mathrm{supp}(\psi(v)) \ (\forall v \in V(P_n^c))$. 
In particular, $\iota(P_n^c) \subset \mathrm{supp}(\psi)$. 
\label{one-component-anti-path}
\end{lemma}
\begin{proof}
Let $\{v_1, v_2, \ldots, v_n \}$ be the vertices of $P_n^c$ labelled as illustrated in Figure \ref{label_P_n_c}. 
\begin{figure}
\centering
\includegraphics[width=45mm,bb=9 9 358 434]{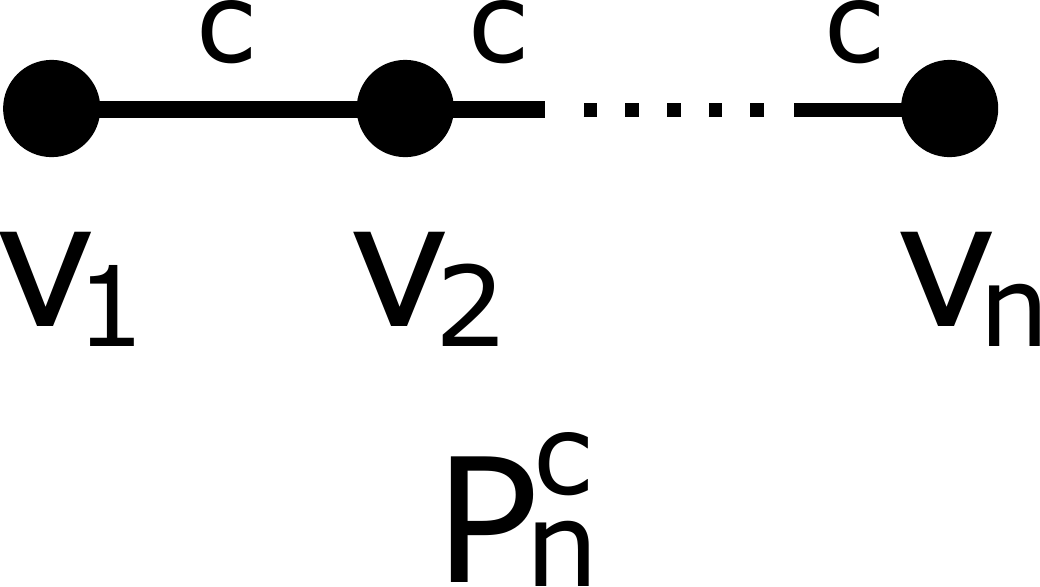}
\caption{This picture illustrates $P_n^c$. 
Real lines with the characters $c$ represent the edges in $P_n = (P_n^c)^c$, each of which joins non-adjacent vertices in $P_n^c$. In this picture, any two distinct vertices not joined by a line are adjacent in $P_n^c$. }
\label{label_P_n_c}
\end{figure}

The assertion is trivial in the case where $n=1$. 
Therefore, we may assume $n = 2$ or $n \geq 4$. 
Suppose $n=2$. 
Then $P_2^c$ consists of two vertices $v_1 , v_2$. 
If there is no full embedding $\iota \colon P_2^c \rightarrow \Gamma$ with $\iota(v_1) \in \mathrm{supp}(\psi(v_1))$ and $\iota(v_2) \in \mathrm{supp}(\psi(v_2))$, then $\mathrm{supp}(\psi(v_1))$ and $\mathrm{supp}(\psi(v_2))$ do not have distinct vertices $u_1 \in \mathrm{supp}(\psi(v_1))$ and $u_2 \in \mathrm{supp}(\psi(v_2))$ such that $\{ u_1 , u_2 \} \not\in E(\Gamma)$. 
Hence the supports, $\mathrm{supp}(\psi(v_1))$ and $\mathrm{supp}(\psi(v_2))$, are contained in a single clique in $\Gamma$, and so $\psi(v_1)$ and $\psi(v_2)$ are commutative in $A(\Gamma)$. 
This implies that a non-trivial element $[v_1, v_2]:= v_1 v_2 v_1^{-1} v_2^{-1}$ of $A(P_2^c)=F_2$ is an element of the kernel of $\psi$, a contradiction. 

We now assume $n \geq 4$. 
By $C_{i}$, we denote the clique in $\Gamma$ spanned by $\mathrm{supp}(\psi(v_i))$. 
Since $\psi(v_i)$ and $\psi(v_j)$ are commutative when  $|i-j| > 1$, we obtain the following claim by Lemma \ref{commutative}. 

\begin{claim}
If $|i -j| > 1$, then any vertex of $C_i$ and any vertex of $C_j$ are either identical or adjacent in $\Gamma$. 
\label{claim-1}
\end{claim}

If $\Gamma$ has a sequence of mutually distinct vertices $y^{(1)}, y^{(2)}, \ldots, y^{(n)}$ such that $y^{(i)} \in V(C_i)$ and that $y^{(i-1)}$ and $y^{(i)}$ are non-adjacent, then the map $\iota \colon P_n^c \rightarrow \Gamma$ defined by $\iota(v_i):= y^{(i)} \ (1 \leq i \leq n)$ determines an embedding $P_n^c \rightarrow \Gamma$ by Claim \ref{claim-1}. 
Since $y^{(i-1)}$ and $y^{(i)}$ are non-adjacent, $\iota \colon P_n^c \rightarrow \Gamma$ is a full embedding. 
Therefore we have only to prove that $\Gamma$ has a sequence of mutually distinct vertices $y^{(1)}, y^{(2)}, \ldots, y^{(n)}$ such that $y^{(i)} \in V(C_i)$ and that $y^{(i-1)}$ and $y^{(i)}$ are non-adjacent. 

Suppose, on the contrary, that 
\begin{itemize}
 \item[($**$)] the graph $\Gamma$ does not have a sequence of mutually distinct vertices $y^{(1)}, y^{(2)}, \ldots, y^{(n)}$ such that $y^{(i)} \in V(C_i)$ and that $y^{(i-1)}$ and $y^{(i)}$ are non-adjacent. 
\end{itemize} 
To deduce a contradiction, we will prove that the commutator $[(v_1)^{v_2 v_3\cdots v_{n-1}}, v_n]$ is a non-trivial element of the kernel of $\psi$. 
We first observe that $[(v_1)^{v_2 v_3\cdots v_{n-1}}, v_n]$ is non-trivial in $A(P_n^c)$. 
Since the preceding letter and succeeding letter of each letter $v^{\epsilon}$ in the word $[(v_1)^{v_2 v_3\cdots v_{n-1}}, v_n]$ are not commutative with $v^{\epsilon}$, the word $[(v_1)^{v_2 v_3\cdots v_{n-1}}, v_n]$ does not admit reduction in the sense of Lemma \ref{reduced}. 
This implies that the word $[(v_1)^{v_2 v_3\cdots v_{n-1}}, v_n]$ is reduced and a non-trivial element in $A(P_n^c)$. 
Thus the remaining task is to show that $[(v_1)^{v_2 v_3\cdots v_{n-1}}, v_n] \in \mathrm{ker}\psi$. 
To this end, it is enough to prove that the element ${\psi(v_1)}^{\psi(v_2)  \psi(v_3) \cdots \psi(v_{n-1})}$ can be represented as a word consisting of vertices of $C_1, \ldots, C_{n-1}$, each of which is commutative with all of the vertices of $C_n$. 
We first inductively define the family $Y^{(1)}, Y^{(2)}, \ldots, Y^{(n-1)}$ of (possibly empty) subsets of $V(\Gamma)$ as follows: 

\begin{enumerate}
 \item[(i)] $Y^{(1)}:= V(C_1)$. 
 \item[(ii)] Suppose that $Y^{(i-1)}$ is defined. 
 Then we set 
 $$Y^{(i)}:= \{ y \in V(C_i) \ | \ \exists x_y \in Y^{(i-1)} \mbox{ such that } \{ x_y , y \} \not\in E(\Gamma) \}. $$ 
\end{enumerate}

With regard to this family $Y^{(1)}, Y^{(2)}, \ldots, Y^{(n-1)}$, we claim that: 

\begin{claim}
Any vertex of $C_n$ and any vertex in $\cup_{i=1}^{n-1} Y^{(i)}$ are either identical or adjacent in $\Gamma$. 
\label{claim-2}
\end{claim}
\begin{proof}[Proof of Claim \ref{claim-2}]
Note that any vertex of $C_n$ and any vertex in $\cup_{i=1}^{n-2} Y^{(i)}$ are either identical or  adjacent in $\Gamma$ by Claim \ref{claim-1}. 
So we have to show that if $Y^{(n-1)} \neq \emptyset$, any element $y^{(n-1)} \in Y^{(n-1)}$ is commutative with all vertices of $C_n$. 
By the construction of $Y^{(n-1)}$, there is an element $y^{(n-2)} \in Y^{(n-2)}$ which is not commutative with $y^{(n-1)}$. 
By repeating this argument, we can find a sequence $y^{(1)}, \ldots, y^{(n-1)}$ such that $y^{(i)} \in Y^{(i)}$, and that $y^{(i-1)}$ and $y^{(i)}$ are not commutative ($2 \leq i \leq n-1$). 
Suppose, on the contrary, that there is an element $y^{(n)} \in V(C_n)$, which is not commutative with $y^{(n-1)}$. 
Then we can observe that $y^{(i)} \neq y^{(j)}$ if $i < j$ as follows. 
We first consider the case where $i=1$ and $2 \leq j \leq n-1$. 
Notice that the element $y^{(1)}$ is either identical with $y^{(j+1)}$ or adjacent to $y^{(j+1)}$ by Claim \ref{claim-1}. 
On the other hand, the element $y^{(j)}$ is non-adjacent to $y^{(j+1)}$. 
Hence, we obtain that $y^{(1)} \neq y^{(j)}$. 
We next consider the case where $i=1$ and $j=n$. 
Since $n\geq 4$, we have $n-1 > 2$. 
Although the element $y^{(n)}$ is non-adjacent to $y^{(n-1)}$, the element $y^{(1)}$ must be adjacent to $y^{(n-1)}$ by Claim \ref{claim-1} and the case where $i=1$ and $j=n-1$. 
Therefore $y^{(1)} \neq y^{(n)}$. 
In case $i \geq 2$, the element $y^{(i)}$ is non-adjacent to $y^{(i-1)}$, but the element $y^{(j)}$ is either adjacent to $y^{(i-1)}$ or identical with $y^{(i-1)}$. 
So $y^{(i)} \neq y^{(j)}$. 
Thus $y^{(i)} \neq y^{(j)}$ if $i < j$, and therefore the sequence $y^{(1)}, \dots, y^{(n)}$ violates our assumption ($**$). 
\end{proof}

Fix reduced words $W_i$ in $V(C_i)$ representing $\psi(v_i)$ ($1 \leq i \leq n$). 
If two given words $w_1, w_2$ represent the same element in $A(\Gamma)$, then we denote by $w_1 = w_2$. 
If $w_1, w_2$ are identical as words, then we denote by $w_1 \equiv w_2$. 
We inductively construct words $\check{W}_1, \ldots, \check{W}_{n-1}$ satisfying the following conditions. 
\begin{enumerate}
 \item[(W-$1$)] $\check{W}_1$ is a word in $Y^{(1)}$. Namely, the word $\check{W}_1$ consists of the vertices in $Y^{(1)}$. 
 \item[(W-$i$)] $\check{W}_i$ is a word in $Y^{(i)}$. Moreover, 
 $$\check{W}_{i}^{-1} \cdots \check{W}_{2}^{-1} \check{W}_1 \check{W}_2 \cdots \check{W}_{i} = W_{i}^{-1} \cdots W_{2}^{-1} W_1 W_2 \cdots W_{i}$$ 
 as elements in $A(\Gamma)$ ($i \geq 2$). 
\end{enumerate}

Let us start the construction of the words $\check{W}_1, \ldots, \check{W}_{n-1}$.

\begin{enumerate}
 \item[(Step 1)] $\check{W}_1 \equiv W_1$. 
 Obviously $\check{W}_1$ satisfies (W-$1$). 
 \item[(Step 2)] If $W_2$ is a word in $Y^{(2)}$, set $\check{W}_2 \equiv W_2$. 
 Then the word $\check{W}_2$ satisfies (W-$2$). 
 We now suppose that $W_2$ is not a word in $Y^{(2)}$,  i.e., there is a vertex $v \in V(C_2) \setminus Y^{(2)}$ such that the letter $v^{\epsilon}$ ($\epsilon = \pm 1$) is contained in $W_2$. 
We write $W_2 \equiv w_2 v^{\epsilon} w_2'$. 
Then $W_2^{-1} \check{W}_1 W_2 \equiv (w_2')^{-1} (v^{\epsilon})^{-1} w_2 \check{W}_1 w_2 v^{\epsilon} w_2'$. 
Note that the letter $v^{\epsilon}$ is commutative with $w_2$, because $C_2$ is a clique containing $v$. 
By the definition of $Y^{(2)}$, the letter $v^{\epsilon}$ is commutative with $\check{W}_1$. 
Hence, $W_2^{-1} \check{W}_1 W_2 = (w_2')^{-1} (w_2)^{-1} \check{W}_1 w_2 w_2'$ in $A(\Gamma)$. 
If $w_2 w_2'$ is a word in $Y^{(2)}$, we set $\check{W}_2 \equiv w_2 w_2'$. 
If not, then applying the same reduction to $w_2 w_2'$
until we obtain a word $\check{W}_2$ in $Y^{(2)}$. 
Then $\check{W}_2$ satisfies the condition (W-$2$). 

 \item[(Step $i$)] Assume that $\check{W}_1, \ldots, \check{W}_{i-1}$ satisfy the conditions (W-$1$), (W-$2$)$,\ldots,$(W-$(i-1)$), respectively. 
 If $W_i$ is a word in $Y^{(i)}$, set $\check{W}_i \equiv W_i$. 
 Then the word $\check{W}_i$ satisfies (W-$i$). 
 We now suppose that $W_i$ is not a word in $Y^{(i)}$.    
 Since $W_i$ is not a word in $Y^{(i)}$, there is a vertex $v \in V(C_i) \setminus Y^{(i)}$ such that the letter $v^{\epsilon}$ ($\epsilon = \pm 1$) is contained in $W_i$. 
 So we write $W_i \equiv w_i v^{\epsilon} w_i'$. 
 Then we have the following equality: 
 \begin{align*}
 &W_i^{-1} \check{W}_{i-1}^{-1} \cdots \check{W}_{2}^{-1} \check{W}_1 \check{W}_2 \cdots \check{W}_{i-1} W_i \\ 
 \equiv &(w_i')^{-1} (v^{\epsilon})^{-1} (w_i)^{-1} \check{W}_{i-1}^{-1} \cdots \check{W}_{2}^{-1} \check{W}_1 \check{W}_2 \cdots \check{W}_{i-1} w_i v^{\epsilon} w_i'.
 \end{align*}
 Since $\check{W}_1, \ldots, \check{W}_{i-2}$ are words in $Y^{(1)}, \ldots, Y^{(i-2)}$, respectively, the letter $v^{\epsilon}$ is commutative with each of $\check{W}_1, \ldots, \check{W}_{i-2}$ by Claim \ref{commutative}. 
 In addition, since $v \in V(C_i) \setminus Y^{(i)}$ and since $\check{W}_{i-1}$ is a word in $Y^{(i-1)}$, the letter $v^{\epsilon}$ is commutative with $\check{W}_{i-1}$.  
 Furthermore, $v^{\epsilon}$ is commutative with $w_i$, because $C_i$ is a clique. 
 Thus we have: 
  \begin{align*}
 &(w_i')^{-1} (v^{\epsilon})^{-1} (w_i)^{-1} \check{W}_{i-1}^{-1} \cdots \check{W}_{2}^{-1} \check{W}_1 \check{W}_2 \cdots \check{W}_{i-1} w_i v^{\epsilon} w_i' \\
 =& (w_i')^{-1} (w_i)^{-1} \check{W}_{i-1}^{-1} \cdots \check{W}_{2}^{-1} \check{W}_1 \check{W}_2 \cdots \check{W}_{i-1} w_i w_i'. 
 \end{align*}
 If $w_i w_i'$ is a word in $Y^{(i)}$, we set $\check{W}_i \equiv w_i w_i'$. 
If not, then applying the same reduction to $w_i w_i'$
until we obtain a word $\check{W}_i$ in $Y^{(i)}$. 
In the end, $\check{W}_i$ obviously satisfies the condition (W-$i$). 
\end{enumerate}

By Claim \ref{claim-2}, $\check{W}_1, \ldots, \check{W}_{n-1}$ are commutative with $W_n$, which is a representative of $\psi(v_n)$. 
Since ${\psi(v_1)}^{\psi(v_2) \psi(v_3) \cdots \psi(v_{n-1})}$ is a multiplication of $\check{W}_1, \ldots, \check{W}_{n-1}$, it is commutative with $\psi(v_n)$. 
Thus the commutator $[(v_1)^{v_2 v_3\cdots v_{n-1}}, v_n]$ is an element of $\mathrm{ker}\psi$, as desired. 
\end{proof}

To treat the case where $n=3$, we use the following lemma due to Kim--Koberda. 

\begin{lemma}[{\cite[Theorem 5.4]{Kim--Koberda13}}]
Let $\Lambda$ be a finite graph whose right-angled Artin group $A(\Lambda)$ has no center. 
Suppose that $A(\Lambda)$ has an embedding into the direct product $G_1 \times G_2$ of (non-trivial) groups $G_1, G_2$. 
If the natural projections $A(\Lambda) \rightarrow  G_i$ ($i=1,2$) have non-trivial kernels, then $\Lambda$ contains a full subgraph isomorphic to the cyclic graph of length $4$. 
\label{Kim--Koberda-P_3^c}
\end{lemma}

\begin{lemma}
Let $\Gamma$ be a finite graph. 
Suppose that $\psi \colon A(P_3^c) \rightarrow A(\Gamma)$. 
Then there is a full embedding $\iota \colon P_3^c \rightarrow \Gamma$ with $\iota(P_3^c) \subset \mathrm{supp}(\psi)$. 
\label{one-component-anti-path-3}
\end{lemma}
\begin{proof}
For simplicity, we assume that $\mathrm{supp}(\psi) = V(\Gamma)$. 
Suppose, on the contrary, that 
\begin{itemize}
 \item[(A)] $P_3^c$ is not a full subgraph of $\Gamma$. Namely, $P_3$ is not a full subgraph of $\Gamma^c$. 
\end{itemize}
We first prove that $\Gamma^c$ is the disjoint union of finitely many complete graphs. 
Let $C$ be a connected component of $\Gamma^c$. 
If $\# V(C) \leq 2$, then the connectedness of $C$ obviously implies that $C$ is complete. 
So we may assume $\# V(C) \geq 3$. 
Pick two edges $e_1^c$ and $e_2^c$ of $C \leq \Gamma^c$ that share a vertex. 
Then, by our assumption (A), the set $e_1^c \cup e_2^c$ of vertices spans a triangle in $\Gamma^c$. 
In other words, the initial vertex and terminal vertex of any edge-path consisting of three vertices in $C$ is adjacent. 
By repeatedly using this fact, we can verify that, for any edge-path in $C$, the initial vertex is adjacent to the terminal vertex. 
Therefore the connected component $C$ must be complete. 
Thus, $\Gamma^c$ is the disjoint union of finitely many complete graphs, and so $\Gamma$ is the join of finitely many edgeless graphs. 
Hence, $A(\Gamma)$ is the direct product $A_1 \times \cdots \times A_m$ of free groups $A_1, \ldots, A_m$. 
Since $A(P_3^c)$ is not free and since $A_1 \times \cdots \times A_m$ contains an embedded $A(P_3^c)$, the integer $m$ is greater than $1$. 
We now regard $A(\Gamma)$ as the direct product $(A_1 \times \cdots \times A_{m-1}) \times A_m$ of two direct factors, $A_1 \times \cdots \times A_{m-1}$ and $A_m$. 
Let $\pi_{m-1}, \pi_m$ denotes the projections $A(P_3^c) \rightarrow A_1 \times \cdots \times A_{m-1}$ and $A(P_3^c) \rightarrow A_m$, respectively. 
Then since $A(P_3^c)$ is not free, the projection $\mathrm{ker}\pi_m$ must be non-trivial. 
Note that $A(P_3^c) \cong \mathbb{Z} * \mathbb{Z}^2$ has no center. 
If $\mathrm{ker} \pi_{m-1}$ is non-trivial, then by Lemma \ref{Kim--Koberda-P_3^c}, the defining graph $P_3^c$ must have a full subgraph isomorphic to the cyclic graph of length $4$, a contradiction. 
So we may assume that $\pi_{m-1}$ is injective. 
In other words, $A(P_3^c) \hookrightarrow A_1 \times \cdots \times A_{m-1}$. 
Hence, by repeating this argument, we can reduce the number of the direct factors in the target group. 
Finally, we have that $A(P_3^c) \hookrightarrow A_1$, which is impossible. 
\end{proof}

\begin{lemma}
Let $\Lambda$ be the join of the complements of finitely many path graphs and $\Gamma$ a finite graph. 
Suppose that $\psi \colon A(\Lambda) \rightarrow A(\Gamma)$ is an embedding satisfying condition (KK). 
Then there is a full embedding $\iota \colon \Lambda \rightarrow \Gamma$ with $\iota(\Lambda) \subset  \mathrm{supp}(\psi)$. 
\label{general-case}
\end{lemma}
\begin{proof}
Suppose that $\Lambda_1, \ldots, \Lambda_m$ is the irreducible graphs such that $\Lambda= \Lambda_1 * \cdots * \Lambda_m$ and $\Lambda_i \cong P_{n_{i}}^c$ ($1 \leq i \leq m$). 
Then by restricting $\psi$ to each direct factor $A(\Lambda_i)$, we obtain $\psi \colon A(\Lambda_i) \hookrightarrow A(\Gamma)$ with condition (KK). 
Now by Lemma \ref{one-component-anti-path} and \ref{one-component-anti-path-3}, we obtain full embeddings $\iota_i \colon \Lambda_i \rightarrow \Gamma$ with $\iota_i(\Lambda_i) \subset \mathrm{supp}(\psi_i)$ for ($1 \leq i \leq m$). 
Since $\Lambda_1^c, \ldots, \Lambda_m^c$ are path graphs, their connectedness together with Lemma \ref{irr-nonad} implies that $\Lambda_1, \ldots, \Lambda_m$ are irreducible. 
Thus, by applying Proposition \ref{joining-EGP} to $\Lambda= \Lambda_1 * \cdots * \Lambda_m$, we obtain the result that there is a full embedding $\iota \colon \Lambda \rightarrow \Gamma$ with $\iota(\Lambda) \subset \mathrm{supp}(\psi)$. 
\end{proof}

\begin{proof}[Proof of Theorem \ref{linear-forest_lemma} (Theorem \ref{linear-forest_lemma_reph}) ]
Suppose that there is an embedding of the right-angled Artin group $A(\Lambda)$ on the join $\Lambda$ of the complements of finitely many path graphs into the right-angled Artin group $A(\Gamma)$ on a finite graph $\Gamma$. 
By Theorem \ref{Kim--Koberda} due to Kim--Koberda, we have an embedding $\psi \colon A(\Lambda) \hookrightarrow A(\Gamma^e)$ satisfying condition (KK), where $\Gamma^e$ is the extension graph of $\Gamma$. 
Consider the full subgraph $\Gamma'$ of $\Gamma^e$, which is spanned by $\mathrm{supp}(\psi) = \cup_{v \in V(\Lambda)} \mathrm{supp}(\psi(v))$. 
Then we have an embedding $\psi \colon A(\Lambda) \hookrightarrow A(\Gamma')$ satisfying condition (KK). 
Now, by Lemma \ref{general-case}, we have $\Lambda \leq \Gamma'$. 
Thus $\Lambda \leq \Gamma' \leq \Gamma^e$, as desired. 
\end{proof}

\end{document}